\documentclass[oneside,10pt]{article}
\usepackage{tabulary}
\usepackage[b5paper]{geometry}

\usepackage{amsfonts,amsmath,latexsym,amssymb}
\usepackage{theorem} 
\usepackage{mathrsfs,upref}
\usepackage{mathptmx}
\usepackage{jmi}

\usepackage{color}
\usepackage{setspace}
\usepackage{amssymb}
\usepackage{caption}
\usepackage{lipsum}
\usepackage{mwe}
\usepackage{amsmath}
\newtheorem{thm}{Theorem}
\newtheorem{cor}[thm]{Corollary}
\newtheorem{lem}[thm]{Lemma}

\usepackage{graphicx}
\usepackage[font=small,labelfont=bf]{caption}
\usepackage[backref]{hyperref}
\usepackage{comment}

\begin{document}
\title [Positivity of log-derivative of the Riemann $\xi$-function]{On a positivity property of the real part of logarithmic derivative of the Riemann $\xi$-function}

\author[E. Gold\v{s}tein, A. Grigutis]{Edvinas Gold\v{s}tein and Andrius Grigutis}

\address{Andrius Grigutis, Institute of Mathematics,
Faculty of Mathematics and Informatics, Vilnius University,
Naugarduko 24, LT-03225 Vilnius, Lithuania,
\email{andrius.grigutis@mif.vu.lt}}

\address{Edvinas Goldštein, Institute of Mathematics,
Faculty of Mathematics and Informatics, Vilnius University,
Naugarduko 24, LT-03225 Vilnius, Lithuania,
\email{edvinasgoldstein@gmail.com}}

\CorrespondingAuthor{Andrius Grigutis}

\date{21.01.2022}



\thanks{We thank professor Ramūnas Garunkštis for his valuable comments writing the paper}

\begin{abstract}
In this paper we investigate the positivity property of the real part of logarithmic derivative of the Riemann $\xi$-function for
$1/2<\sigma<1$ and sufficiently large $t$. We give an explicit upper and lower bounds for $\Re\sum_{\rho} 1/(s-\rho)$, where the sum runs over the zeros of $\zeta(s)$ on the line $1/2+it$. We also check the positivity of $\Re \xi'/\xi(s)$ for $1/2<\sigma<1$ assuming that there occur a non-trivial zeros of $\zeta(s)$ off the critical line.
\end{abstract}

\maketitle

\section{Introduction}
    For the complex $s=\sigma+it$ the Riemann $\xi$-function is defined by
$$
\xi(s)=\frac{1}{2}s(s-1)\pi^{-s/2}\Gamma(s/2)\zeta(s),
$$
where $\zeta(s)$ is Riemann $\zeta$-function.
The functions $\xi(s)$ and $\zeta(s)$ have the same zeros in the strip $0<\sigma<1$ and the famous
Riemann hypothesis states that they all are located on the line $1/2+it$ - called the critical line.
Zeros in the strip $0<\sigma<1$ are known as non-trivial zeros of $\zeta(s)$.
The Riemann $\zeta$-function also has zeros at each even negative integer $s=-2n$ - these are known as the trivial zeros of $\zeta(s)$.
The function $\xi(s)$ also satisfies $\xi(s)=\xi(1-s)$ and $\overline{\xi(s)}=\xi(\overline{s})$.
From this, it is clear that $\xi(\sigma+it)=0$ iff $\xi(1-\sigma+it)=0$. Also, if $s$ is a non-trivial zero of $\xi(s)$ off the critical line then the four numbers $\{s,\bar{s},1-s,1-\bar{s}\}$ would all be non-trivial zeros off the line.

    By $\rho=\beta+i\gamma$ we denote a non-trivial zero of $\zeta(s)$, i.e. $\zeta(\rho)=0$.
The function $\xi(s)$ can be expanded as an infinite product by $\rho$, see Edwards \cite[p. 39]{Edwards} and Wolfram MathWorld \cite{MathWorld},
\begin{align}\label{eq: xi prod}
\xi(s)=\xi(0)\prod_{\rho}\left(1-\frac{s}{\rho}\right)=\frac{1}{2}\prod_{\rho}\left(1-\frac{s}{\rho}\right),
\end{align}
where the product is taken in an order which pairs each root $\rho$ with the corresponding root $1-\rho$. The logarithmic derivative of $\xi(s)$ is
\begin{align}\label{eq:xi sum}
\frac{\xi^{\prime}}{\xi}(s)=\sum_{\rho}\frac{1}{s-\rho},
\end{align}
where the summation is understood the same way as defining the product (\ref{eq: xi prod}).
There is a direct relation between location of zeros of complex function
$f$ and behavior of its modulus or real part of logarithmic derivative.
Matiyasevich, Saidak and Zvengrowski \cite{MSZ} note that "...strict
decrease of the modulus of any continuous complex function $f$ along any
curve in the complex plane clearly implies that $f$ can have no zero along
that curve." The relation between monotonicity of modulus of complex function $|f|$ and sign of its real part of logarithmic derivative $\Re \frac{f'}{f}$ is provided in Lemma \ref{increas}.

It is known that (see for example Hinkkanen \cite{Hinkkanen})
$$
\Re\frac{\xi'}{\xi}(s)>0 \text{ when } \Re s>1
$$
and the Riemann hypothesis is equivalent to
$$
\Re\frac{\xi'}{\xi}(s)>0 \text{ when } \Re s>\frac{1}{2}.
$$

Lagarias \cite{Lagarias} proved that
\begin{align}\label{Lagarias}
\inf\left\{\Re\frac{\xi'}{\xi}(s):-\infty<t<\infty\right\}=\frac{\xi'}{\xi}(\sigma)
\end{align}
for $\sigma>10$ and Garunk\v{s}tis \cite{RG} later improved (\ref{Lagarias}) for $\sigma >a$, where $\sigma>a$ is a zero-free region of $\zeta(s)$. See also Broughan \cite{Broughan} on the subject.

    In the paper by Sondow and Dumitrescu \cite{Sondow} there was given the following reformulation of the Riemann hypothesis.
\begin{thm}[Sondow, Dumitrescu]
The following statements are equivalent.\\
$\mathrm{I.}$ If $t$ is any fixed real number, then $|\xi(\sigma+it)|$ is increasing for $1/2<\sigma<\infty$.\\
$\mathrm{II.}$ If $t$ is any fixed real number, then $|\xi(\sigma+it)|$ is decreasing for $-\infty<\sigma<1/2$.\\
$\mathrm{III.}$ The Riemann hypothesis is true.
\end{thm}
Also, in the same paper it was proved the following theorem.
\begin{thm}[Sondow, Dumitrescu]\label{S-D}
The $\xi$-function is increasing in modulus along every horizontal half-line lying
in any open right half-plane that contains no $\xi$ zeros. Similarly, the modulus decreases on
each horizontal half-line in any zero-free, open left half-plane.
\end{thm}

Matiyasevich, Saidak and Zvengrowski \cite{MSZ} slightly reformulated the Theorem \ref{S-D}.

\begin{thm}[Matiyasevich, Saidak, Zvengrowski]\label{MSZ}
Let $\sigma_0$ be greater than or equal to the real part of any
zero of $\xi$. Then $|\xi(s)|$ is strictly increasing\footnote{With respect to $\sigma$.} in the half-plane $\sigma>\sigma_0$.
\end{thm}

In this paper we further investigate the function $ \xi'/\xi(s)$. We set
\begin{align}\label{eq:xi split}
\frac{\xi'}{\xi}(s)=\sum_{\rho=1/2+i\gamma}\frac{1}{s-\rho}+ \sum_{\tilde{\rho}\neq1/2+i\gamma}\frac{1}{s-\tilde{\rho}}=:\Sigma_1+\Sigma_2,
\end{align}
where the summation again is understood as defining (\ref{eq: xi prod}). This ensures an absolute convergence of the series in (\ref{eq:xi split}) for $s:\zeta(s)\neq0$, see Edwards \cite[p. 42]{Edwards}. Obviously, the sum $\Sigma_1$ exists, while $\Sigma_2$ might be vacuous as the Riemann hypothesis is unsolved.

For $1/2<\sigma<1$ and sufficiently large $t$, in Theorem \ref{Main thm} below, we give an explicit lower and upper bounds for $\Re\Sigma_1$.
The lower bound of $\Re\Sigma_1$ in Theorem \ref{Main thm} suggests that $\Re \xi'/\xi(s)$ may remain positive asymptotically close to the critical line despite that $\Re\Sigma_2$ might occur if the Riemann hypothesis fails. In Section $\ref{Can?}$ we test the positivity of $\Re(\Sigma_1+\Sigma_2)$ assuming that a certain versions of $\Sigma_2$ exist -
an obtained results widen Theorems $\ref{S-D}$ and \ref{MSZ}, see Figures 1 and 2 in Section \ref{Can?}.

We start the investigation of $\Sigma_1$ by an observation that there are infinitely many zeros of $\zeta(s)$ lying on the line $1/2+it$ (see Hardy \cite{Hardy}), however we do not know the quantity of zeros of $\zeta(s)$ lying in the strip $1/2<\sigma<1$.
The initial result on the part of non-trivial zeros on the critical line of the Riemann zeta-function was obtained by Selberg \cite{Selberg}. There was proved that at least a positive proportion of all non-trivial zeros lie on the critical line. Later this result was improved by several authors,
see for example Levinson \cite{Levinson}, Conrey \cite{Conrey}, Feng \cite{Feng}, Pratt et al. \cite{Pratt et al.}. Based on the mentioned facts, we formulate the following theorem for $\Re\Sigma_1$.

\begin{thm}\label{Main thm}
Let $1/2<\sigma<1$. Let $c$ be the part of non-trivial zeros of $\zeta(s)$ lying on the line $1/2+it$ and
\begin{align*}
&A(t)=0.12\log\frac{t}{2\pi}-2.32\log\log t-18.432-\epsilon_1(t),\\
&B(t)=0.49\log\frac{t}{2\pi}+0.58\log\log t+4.603+\epsilon_2(t),
\end{align*}
where $\epsilon_1(t)$ and $\epsilon_2(t)$ are known explicit $t$ functions (see (\ref{eq:eps1}) and (\ref{eq:eps2}) below) both vanishing as $t^{-1}\log t,\,t\to\infty$.

Then
\begin{align*}
&0<c\left(\sigma-\frac{1}{2}\right)A(t)
<\Re\sum_{\rho=1/2+i\gamma}\frac{1}{s-\rho},\,t>1.984\times10^{114},\\
&\qquad \qquad \qquad \qquad \quad \, \, \, \Re\sum_{\rho=1/2+i\gamma}\frac{1}{s-\rho}<\frac{cB(t)}{\sigma-1/2},\,t>14.635.
\end{align*}

\end{thm}

We prove Theorem \ref{Main thm} in Section \ref{proof}. This theorem leads to the following corollary.

\begin{cor}\label{cor}
The function
$$
\Re\frac{\xi'}{\xi}(s)=-\Re\frac{\xi'}{\xi}(1-s)>0
$$
if
\begin{align}\label{ineq:pos_cond}
\Re\sum_{\tilde{\rho}\neq1/2+i\gamma}\frac{1}{s-\tilde{\rho}}+c\left(\sigma-\frac{1}{2}\right)A(t)>0.
\end{align}
\end{cor}

The remaining structure of this article is: in Section \ref{lemmas} we formulate an auxiliary statements, while in the last Section \ref{Can?} we depict the condition (\ref{ineq:pos_cond}) assuming that the Riemann hypothesis fails. 

\section{Lemmas}\label{lemmas}

In this section we formulate a several auxiliary lemmas, which are needed for the proof of Theorem \ref{Main thm}.

\begin{lem}\label{increas}
$\mathrm(a)$ Let $f$ be holomorphic in an open domain $D$ and
not identically zero. Let us also suppose
$\Re\left(f^{\prime}(s)/f(s)\right)<0$
for all $s\in D$ such that $f(s)\neq 0$.
Then $|f(s)|$ is strictly decreasing with respect to $\sigma$ in
$D$, i.e. for each $s_0 \in D$ there exists a $\delta > 0$ such that $|f(s)|$ is strictly
monotonically decreasing with respect to $\sigma$ on the horizontal interval from
$s_0 - \delta$ to $s_0 + \delta$.\\
$\mathrm(b)$ Conversely, if $|f(s)|$ is decreasing with respect to $\sigma$ in $D$, then
$\Re\left(f^{\prime}(s)/f(s)\right)\leqslant 0$ for all $s\in D$ such that $f(s) \neq 0$.
\end{lem}
\begin{proof}
See Matiyasevich, Saidak, Zvengrowski \cite{MSZ} for the proof.
\end{proof}

\bigskip

{\sc Note 1:} Of course, the analogous results hold for monotone increasing $|f(s)|$ and
$\Re\left(f^{\prime}(s)/f(s)\right)>0$.

\begin{lem}\label{number of zeros}
Let $N(T)$ be the number of zeros of $\zeta(s)$ in the rectangle $0<\sigma<1, \\ 0<t<T$. If $T\geqslant e$, then
\begin{align}\label{number of zeros estimate}
\left|N(T)-\frac{T}{2\pi} \log \frac{T}{2\pi e}-\frac{7}{8}\right|\leqslant 0.110\log T + 0.290\log \log T+2.290+\frac{25}{48\pi T}.
\end{align}
\end{lem}
\begin{proof}
In the paper by Trudgian \cite[p. 283]{Trudgian} it is derived that, for $T\geqslant1$
\begin{align*}
\left|N(T)-\frac{T}{2\pi} \log \frac{T}{2\pi e}-\frac{7}{8}\right|\leqslant
|S(T)|+\frac{1}{4\pi}\arctan\left(\frac{1}{2T}\right)+\frac{T}{4\pi}\log\left(1+\frac{1}{4T^2}\right)+\frac{1}{3\pi T},
\end{align*}
where $\pi S(T)$ is the argument of the Riemann zeta-function along the critical line. From the paper by Platt and Trudgian \cite[Cor. 1]{Platt_Trudgian} (see also Hasanalizade, Shen, Wong \cite{Has et al.})
$$
|S(T)|\leqslant 0.110\log T + 0.290\log \log T+2.290,\,T\geqslant e
$$
and, using inequalities,
$$
\arctan\frac{1}{t}=\int_{0}^{1/t}\frac{dx}{1+x^2}\leqslant \frac{1}{t},\,t>0 \\
$$
and
$$
\log(1+t)\leqslant t,\, t>-1,
$$
we get a desired result. 
\end{proof}

\begin{lem}\label{lem:lower_bound}
If $a,\,b,\,\alpha>0$, then the following inequality holds
$$
\int_{\alpha}^{t}\frac{\log\frac{u}{2\pi}\,du}{a^2+b^2(u-t)^2}\geqslant \frac{1}{ab}\log\left(\frac{t}{2\pi}\right)\arctan\left(\frac{b(t-\alpha)}{a}\right)-\kappa,
$$
when $t>t_0\geqslant \alpha$, where $t_0$ and constant $\kappa>0$ are both sufficiently large and $\kappa$ is independent on $t$.

    In particular, if $a=1/2,\,b=1$ and $\alpha=14.134725\ldots$, then the provided inequality holds if $t>23$ and $\kappa=0.135$.
\end{lem}
\begin{proof}
We set up the function 
$$
F(t)=\int_{\alpha}^{t}\frac{\log\frac{u}{2\pi}\,du}{a^2+b^2(u-t)^2}- \frac{1}{ab}\log\left(\frac{t}{2\pi}\right)\arctan\left(\frac{b(t-\alpha)}{a}\right)+\kappa
$$
and show that $t$ derivative $F'(t)\geqslant 0$ for $t>t_0\geqslant\alpha$. Indeed, according to the Leibniz integral rule (see for example Mackevičius \cite{Vigirdas} or Spivak \cite{Spivak})
\begin{align*}
F'(t)&=2b^2\int_{\alpha}^{t}\frac{(u-t)\log u/2\pi\,du}{(a^2+b^2(u-t)^2)^2}
+\left(\frac{1}{a^2}-\frac{1}{a^2+b^2(t-\alpha)^2}\right)\log \frac{t}{2\pi}\\
&-\frac{\arctan(b(t-\alpha)/a)}{abt}.
\end{align*}
The last integral is
\begin{align*}
&2b^2\int_{\alpha}^{t}\frac{(u-t)\log u/2\pi\,d u}{(a^2+b^2(u-t)^2)^2}
=-\int_{\alpha}^{t}\log\frac{u}{2\pi}\,d\,\frac{1}{a^2+b^2(u-t)^2}\\
&=\frac{\log (\alpha/2\pi)}{a^2+b^2(t-\alpha)^2}
-\frac{\log (t/2\pi)}{a^2}+\int_{\alpha}^{t}\frac{d u}{u(a^2+b^2(t-\alpha)^2)},
\end{align*}
where 
\begin{align*}
&\int_{\alpha}^{t}\frac{d u}{u(a^2+b^2(t-\alpha)^2)}
=\frac{b^2}{b^2t^2+a^2}\int_{\alpha}^{t}\left(\frac{1}{b^2u}+\frac{2t-u}{a^2+b^2(u-t)^2}\right)du\\
&=\frac{\log (t/\alpha)}{b^2t^2+a^2}
+\frac{b}{a}\cdot\frac{t}{b^2t^2+a^2}\arctan\left(\frac{b(t-\alpha)}{a}\right)
+\frac{1}{2}\cdot\frac{1}{b^2t^2+a^2}\log\left(1+\frac{b^2(t-\alpha)^2}{a^2}\right).
\end{align*}
Therefore
\begin{align*}
&F'(t)=\frac{1/2}{b^2t^2+a^2}\log \left(\left(\frac{t}{\alpha}\right)^2+\left(\frac{b t(t-\alpha)}{a\alpha}\right)^2\right)
-\frac{\log(t/\alpha)}{a^2+b^2(t-\alpha)^2}
\\
&-\frac{a}{b}\cdot\frac{1}{t}\cdot\frac{1}{b^2t^2+a^2}\arctan\left(\frac{b(t-\alpha)}{a}\right).
\end{align*}
For $t\geqslant\alpha+a/b$, it holds that
$$
\frac{b t(t-\alpha)}{a\alpha}\geqslant\frac{t}{\alpha}, 
$$
and 
\begin{align}\label{ineq:deriv_estimate}
F'(t)
&\geqslant\frac{\log\sqrt{2}}{a^2+b^2t^2}-\frac{(\alpha(2t-\alpha))\log(t/\alpha)}{(a^2+b^2t^2)(a^2+b^2(t-\alpha)^2)}\\
&-\frac{a}{b}\cdot\frac{1}{t}\cdot\frac{1}{b^2t^2+a^2}
\arctan\left(\frac{b(t-\alpha)}{a}\right).\nonumber
\end{align}
The positive term of the right-hand side of inequality (\ref{ineq:deriv_estimate})
vanishes as $t^{-2}$ while the two negative terms as $t^{-3}\log t$, which means that $F'(t)>0$ if $t>t_0\geqslant\alpha$ and $t_0$ is sufficiently large.

We next check whether $F(t_0)\geqslant0$. It is easy to see that
$$
\lim_{t\to\alpha^+}F(t)=\kappa>0.
$$
Therefore, due to continuity of $F(t)$, $F(t)>0$ for at least $t\in(\alpha,t_0]$ if $\kappa$ is large enough and $t_0$ is dependent on $\kappa$.

For the particular case $a=1/2,\,b=1$ and $\alpha=14.134725\ldots$ we check with Mathematica \cite{Mathematica} that $F'(t)>0$, when $t>23$ and $F(23)=0.00092\ldots$ if $\kappa=0.135$.

\end{proof}
\begin{lem}\label{arctan}
If $t>1$, then
\begin{align}\label{arctan_ineq}
\frac{\pi}{2}-\frac{1}{t}<\arctan t <\frac{\pi}{2}-\frac{1}{2t}.
\end{align}
\end{lem}

\begin{proof}
The first inequality of (\ref{arctan_ineq}) follows from
$$
\frac{\pi}{2}=\int_{0}^{\infty}\frac{dx}{1+x^2}=\int_{0}^{t}\frac{dx}{1+x^2}+\int_{t}^{\infty}\frac{dx}{1+x^2}
<\arctan t + \int_{t}^{\infty}\frac{dx}{x^2}=\arctan t + \frac{1}{t},
$$
and the second
$$
\frac{\pi}{2}=\int_{0}^{\infty}\frac{dx}{1+x^2}=\int_{0}^{t}\frac{dx}{1+x^2}+\int_{t}^{\infty}\frac{dx}{1+x^2}
>\arctan t + \int_{t}^{\infty}\frac{dx}{x^2+x^2}=\arctan t + \frac{1}{2t}.
$$
{\sc Note 2}: The first inequality in (\ref{arctan_ineq}) holds for $t>0$ also.\\
{\sc Note 3}: The function $\arctan$ is an odd function and for $t<-1$ the estimates are
$-\frac{\pi}{2}-\frac{1}{2t}<\arctan(t)<-\frac{\pi}{2}-\frac{1}{t}$.
\end{proof}

\begin{lem}\label{int}
Let $\alpha>0$ and $b>a>0$ be a constants. For $t>t_0\geqslant \alpha+a/b$, let
$$
\tilde{A}(t):=\frac{\pi}{ab}\log\left(\frac{t}{2\pi}\right)-\frac{\log\frac{t}{2\pi}}{b^2(t-\alpha)}-\kappa
$$
and
$$
\tilde{B}(t):=\left(\frac{\pi}{ab}+\frac{1}{2b^2}\right)\log\frac{t+1}{2\pi}
+\frac{\log(t+1)}{b^2t},
$$
where $\kappa>0$ is a constant from Lemma \ref{lem:lower_bound} and $t_0$ is sufficiently large.

Then
$$
\tilde{A}(t)<\int_{\alpha}^{\infty}\frac{\log(u/2\pi)\, du}{a^2+b^2(u-t)^2}<\tilde{B}(t).
$$
\end{lem}
\begin{proof}
For the lower bound, by elementary calculation and Lemmas \ref{lem:lower_bound} and \ref{arctan},  we obtain
\begin{align*}
&\int_{\alpha}^{\infty}\frac{\log(u/2\pi)\, du}{a^2+b^2(u-t)^2}
=\left(\int_{\alpha}^{t}+\int_{t}^{\infty}\right)\frac{\log(u/2\pi)\, du}{a^2+b^2(u-t)^2}\\
&>\int_{\alpha}^{t}\frac{\log(u/2\pi)\,d u}{a^2+b^2(u-t)^2}
+\log\left(\frac{t}{2\pi}\right)\int_{t}^{\infty}\frac{d u}{a^2+b^2(u-t)^2}\\
&>\frac{1}{ab}\log\left(\frac{t}{2\pi}\right)\arctan\left(
\frac{b(t-\alpha)}{a}\right)-\kappa
+\frac{\pi/2}{ab}\log\left(\frac{t}{2\pi}\right)\\
&>\frac{\pi}{ab}\log\left(\frac{t}{2\pi}\right)-\frac{\log\frac{t}{2\pi}}{b^2(t-\alpha)}-\kappa=\tilde{A}(t).
\end{align*}
By the same thoughts for the upper bound we get
\begin{align*}
&\int_{\alpha}^{\infty}\frac{\log(u/2\pi)\, du}{a^2+b^2(u-t)^2}
=\left(\int_{\alpha}^{t+1}+\int_{t+1}^{\infty} \right)\frac{\log(u/2\pi)\, du}{a^2+b^2(u-t)^2}\\
&<\log\left(\frac{t+1}{2\pi}\right)\int_{\alpha}^{t+1}\frac{du}{a^2+b^2(u-t)^2}
+\frac{1}{b^2}\int_{t+1}^{\infty}\frac{\log(u/2\pi)\, du}{(u-t)^2}\\
&=\frac{1}{ab}\log\left(\frac{t+1}{2\pi}\right)\left(\arctan\left(\frac{b}{a}\right)
+\arctan\left(\frac{t-\alpha}{a/b}\right)\right)
+\frac{\left(1+\frac{1}{t}\right)\log(t+1)-\log2\pi}{b^2}\\
&<\left(\frac{\pi}{ab}-\frac{t-\alpha+1}{2b^2(t-\alpha)}\right)\log\left(\frac{t+1}{2\pi}\right)
+\frac{\left(1+\frac{1}{t}\right)\log(t+1)-\log2\pi}{b^2}\\
&<\left(\frac{\pi}{ab}+\frac{1}{2b^2}\right)\log\frac{t+1}{2\pi}
+\frac{\log(t+1)}{b^2t}
=\tilde{B}(t).
\end{align*}
\end{proof}

\begin{lem}\label{int1}
Let $\alpha>0$ and $b>a\geqslant0$ be a constants. For $t>\alpha+a/b$, let
$$
\tilde{C}(t):=\frac{1}{4b^2 t}\log\left(\frac{t}{2\pi}\right)-\frac{\alpha}{b^2t^2}\log \left(\frac{\alpha}{2 \pi}\right)
$$
and
$$
\tilde{D}(t):=\frac{1}{2b^2 t}\log\left(\frac{2t^3}{4\pi^3}\right).
$$
Then
$$
\tilde{C}(t)<\int_{\alpha}^{\infty}\frac{\log(u/2\pi)\, du}{a^2+b^2(u+t)^2}<\tilde{D}(t).
$$
\end{lem}
\begin{proof}
We do the same as in the proof of the previous lemma.
For the lower bound

\begin{align*}
&\int_{\alpha}^{\infty}\frac{\log(u/2\pi) \, du}{a^2+b^2(u+t)^2}=
\left(\int_{\alpha}^{t}+\int_{t}^{\infty}\right)\frac{\log(u/2\pi) \, du}{a^2+b^2(u+t)^2}\\
&>\frac{1}{ab} \log\left( \frac{\alpha}{2 \pi}\right) \left(\arctan\frac{2t}{a/b} - \arctan\frac{t+\alpha}{a/b}\right)+
\frac{1}{ab}\log \left(\frac{t}{2\pi}\right)\left( \frac{\pi}{2} - \arctan\frac{2t}{a/b} \right)\\
&>\frac{1}{ab}\log \left(\frac{\alpha}{2 \pi}\right)\left(\frac{\pi}{2}-\frac{a/b}{2t} - \frac{\pi}{2}+\frac{a/b}{2(t+\alpha)}\right)
+\frac{1}{ab}\log \left(\frac{t}{2\pi}\right)\left( \frac{\pi}{2} -\frac{\pi}{2}+\frac{a/b}{4t} \right)\\
&>\frac{1}{4b^2 t}\log\left(\frac{t}{2\pi}\right)-\frac{\alpha}{b^2t^2}\log \left(\frac{\alpha}{2 \pi}\right)
=\tilde{C}(t).
\end{align*}

And for the upper bound

\begin{align*}
&\int_{\alpha}^{\infty}\frac{\log(u/2\pi) \, du}{a^2+b^2(u+t)^2}=
\left(\int_{\alpha}^{t}+\int_{t}^{\infty}\right)\frac{\log(u/2\pi) \, du}{a^2+b^2(u+t)^2}\\
&<\log \left(\frac{t}{2 \pi}\right)\int_{\alpha}^{t}\frac{du}{a^2+b^2(u+t)^2}
+\int_{t}^{\infty}\frac{\log(u/2\pi) \, du}{b^2(u+t)^2}\\
&=\frac{1}{ab}\log \left(\frac{t}{2 \pi}\right)\left(\arctan\left(\frac{2t}{a/b}\right)-\arctan\left(\frac{t+\alpha}{a/b}\right)\right)
+\frac{1}{2b^2t}\log\left(\frac{2t}{\pi}\right)\\
&<\frac{1}{ab}\log \left(\frac{t}{2 \pi}\right)\left(\frac{\pi}{2} -\frac{a/b}{4t}-\frac{\pi}{2}+\frac{a/b}{t+\alpha}\right)+\frac{1}{2b^2t}\log\left(\frac{2t}{\pi}\right)\\
&<\frac{1}{2b^2 t}\log\left(\frac{2t}{\pi}\right)+\frac{1}{b^2t}\log\left(\frac{t}{2\pi}\right)=\tilde{D}(t).
\end{align*}
\end{proof}

The next lemma we need is well known as a summation by parts.
\begin{lem}\label{sumation by parts}
Let $\{a_n\}_{n=1}^{\infty}$ be a sequence of complex numbers and $G(u)$ a continuously differentiable function on $[1,x]$. If $A(u)=\sum_{n\leqslant u}a_n$,
then
$$
\sum_{n\leqslant x}a_n G(n)=A(x)G(x)-\int_{1}^{x}A(u)G'(u)\,du.
$$
\end{lem}

\begin{proof}
See for example Murty \cite[p. 18]{Murty} or Apostol \cite[p. 54]{Apostol} for the proof.
\end{proof}

\bigskip

In the below met inequalities numbers are rounded up to two or three decimal places. 

\begin{lem}\label{sumzero}
Let $\rho=\beta+i\gamma$ denote a non-trivial zero of $\zeta(s)$. Let $a,b>0$ and $\gamma_1:= 14.134725 \ldots$ $(\zeta(1/2+i\gamma_1)=0)$. If $t>\gamma_1$, then
\begin{align*}
&\sum_{\rho=\beta+i\gamma}\frac{1}{a^2+b^2(t-\gamma)^2}=
\sum_{\gamma>0}\frac{1}{a^2+b^2(t-\gamma)^2}+\sum_{\gamma>0}\frac{1}{a^2+b^2(t+\gamma)^2}=:S_1+S_2,
\end{align*}
where
\begin{align*}
\left|S_1-\frac{1}{2\pi}\int_{\gamma_1}^{\infty}\frac{\log(u/2\pi) \, d u}{a^2+ b^2(u-t)^2}\right|
<&\,\frac{0.22\log t+0.58\log\log t+4.58}{a^2}\\
&+\frac{0.166}{a^2t}\left(1+\frac{2.411a}{b}\right)
\end{align*}
and
\begin{align*}
\left|S_2-\frac{1}{2\pi}\int_{\gamma_1}^{\infty}\frac{\log(u/2\pi) \, du}{a^2+ b^2(u+t)^2}\right|<\frac{3.811}{a^2+b^2(\gamma_1+t)^2}+\frac{0.045}{ab}.
\end{align*}

\end{lem}
\begin{proof}
Since $\zeta(\rho)=\zeta(\bar{\rho})=0$ we have that
$$
\sum_{\rho=\beta+i\gamma}\frac{1}{a^2+b^2(t-\gamma)^2}=\sum_{\gamma>0}\frac{1}{a^2+b^2(t-\gamma)^2}+\sum_{\gamma>0}\frac{1}{a^2+b^2(t+\gamma)^2}=S_1+S_2.
$$
For $S_1$, by Lemma \ref{sumation by parts},
\begin{align*}
S_1=-\int_{\gamma_1}^{\infty}N(u)f'(u)du,
\end{align*}
where $f(u):=1/(a^2+b^2(t-u)^2)$ and the step function $N(u)$ is defined in Lemma \ref{number of zeros}. Let $N_{up}(u)$ and $N_{low}(u)$ be the corresponding continues upper and lower bounds of $N(u)$. From Lemma \ref{number of zeros}
\begin{align*}
&N_{up}(u)=\frac{u}{2\pi}\log\frac{u}{2\pi e}+0.11\log u+0.29\log \log u+3.165+\frac{25}{48\pi u},\\
&N_{low}(u)=\frac{u}{2\pi}\log\frac{u}{2\pi e}-0.11\log u-0.29\log \log u-1.415-\frac{25}{48\pi u}.
\end{align*}

Let us observe that $u$ derivative $f'(u)$ is non-negative for $u\leqslant t$ and $f'(u)$ is negative for $u>t$. As $N_{up}(u), \, N_{low}(u)$ are continues functions, then

\begin{align}\label{I_1}
\nonumber
&S_1\leqslant-\int_{\gamma_1}^{t}N_{low}(u)f'(u)du-\int_{t}^{\infty}N_{up}(u)f'(u)du
=-\int_{\gamma_1}^{\infty}\frac{u}{2\pi}\log\frac{u}{2\pi e}f'(u)du\\
\nonumber
&+\int_{\gamma_1}^{t}\left(0.11\log u+0.29\log\log u+1.415+\frac{25}{48\pi u}\right)d\,f(u)\\
\nonumber
&-\int_{t}^{\infty}\left(0.11\log u+0.29\log\log u+3.165+\frac{25}{48\pi u}\right)d\,f(u)
\\
\nonumber
&\leqslant\frac{1}{2\pi}\int_{\gamma_1}^{\infty}\frac{\log(u/2\pi) \, du}{a^2+b^2(u-t)^2}+\frac{\gamma_1}{2\pi}\log\left(\frac{\gamma_1}{2\pi e}\right)f(\gamma_1)\\
\nonumber
&+(f(t)-f(\gamma_1))\left(0.11\log t+0.29\log\log t+1.415+\frac{25}{48\pi \gamma_1}\right)\\
&+f(t)\left(3.165+\frac{25}{48\pi t}\right)-0.11\int_{t}^{\infty}\log u \,d\,f(u)
-0.29\int_{t}^{\infty}\log \log u\, d\,f(u).
\end{align}
For the integrals in (\ref{I_1}) it holds that
\begin{align*}
&-\int_{t}^{\infty}\log u\, d\,f(u)=f(t)\log t+\int_{t}^{\infty}\frac{f(u)du}{u}<f(t)\log t+\frac{\pi/2}{ab}\cdot\frac{1}{t},\\
&-\int_{t}^{\infty}\log \log u \, d\,f(u)
<f(t)\log \log t+\frac{\pi/2}{ab}\cdot\frac{1}{t\log t}.\\
\end{align*}

Therefore
\begin{align*}
S_1&<\frac{1}{2\pi}\int_{\gamma_1}^{\infty}\frac{\log(u/2\pi) \, d u}{a^2+b^2(u-t)^2}+\frac{0.220\log t+0.580\log\log t+4.580}{a^2}
\\
&+\frac{0.166}{a^2t}\left(1+\frac{2.413a}{b}\right).
\end{align*}

By the similar arguments, the lower bound of $S_1$ is
\begin{align*}
S_1&>\frac{1}{2\pi}\int_{\gamma_1}^{\infty}\frac{\log(u/2\pi) \, du}{a^2+b^2(u-t)^2}-\frac{0.220\log t+0.580\log\log t+4.580}{a^2}
\\
&-\frac{0.166}{a^2t}\left(1+\frac{2.413a}{b}\right).
\end{align*}

The upper bound of
$$
S_2=-\int_{\gamma_1}^{\infty}N(u)g'(u)du,\,g(u):=1/(a^2+b^2(t+u)^2),
$$
observing that $g(u)$ is decreasing for $u\geqslant0$, is
\begin{align}
\nonumber
S_2<&-\int_{\gamma_1}^{\infty}N_{up}(u)g'(u)\,du=-\int_{\gamma_1}^{\infty}\frac{u}{2\pi}\log\frac{u}{2\pi e}\,d\,g(u)\\
\nonumber
&-\int_{\gamma_1}^{\infty}\left(0.11\log u+0.29\log\log t+3.165+\frac{25}{48\pi u}\right)g'(u)\, d u\\
\nonumber
&=\frac{1}{2\pi}\int_{\gamma_1}^{\infty}\frac{\log \left(u/2\pi\right)\,du}{a^2+b^2(u+t)^2}+
\frac{\gamma_1}{2\pi}\log\left(\frac{\gamma_1}{2\pi e}\right)g(\gamma_1)\\
&-0.11\int_{\gamma_1}^{\infty}\log u \,g'(u)\,d u
-0.29\int_{\gamma_1}^{\infty}\log \log u\, g'(u)\,d u
\label{Int_1}
\\
&-\int_{\gamma_1}^{\infty}\left(3.165+\frac{25}{48\pi u}\right)g'(u)\,du.
\label{Int_2}
\end{align}

The integrals in (\ref{Int_1}) and (\ref{Int_2}) evaluate to 
\begin{align*}
&-\int_{\gamma_1}^{\infty}\log u\, g'(u)\,du=g(\gamma_1)\log \gamma_1+\int_{\gamma_1}^{\infty}\frac{du}{u(a^2+b^2(t+u)^2)}
<g(\gamma_1)\log \gamma_1+\frac{\pi/2}{\gamma_1ab},\\
&-\int_{\gamma_1}^{\infty}\log \log u\, g'(u)\, d u<g(\gamma_1)\log \gamma_1+\frac{\pi/2}{\gamma_1ab},\\
&-\int_{\gamma_1}^{\infty}\left(3.165+\frac{25}{48\pi u}\right)g'(u)\,du<g(\gamma_1)\left(3.165+\frac{25}{48\pi\gamma_1}\right).
\end{align*}

Therefore 
$$
S_2<\frac{1}{2\pi}\int_{\gamma_1}^{\infty}\frac{\log \left(u/2\pi\right)\,du}{a^2+b^2(u+t)^2}+3.811 g(\gamma_1)+\frac{0.045}{ab}.
$$

Arguing the same, the lower bound of $S_2$ is
$$
S_2>\frac{1}{2\pi}\int_{\gamma_1}^{\infty}\frac{\log \left(u/2\pi\right)\,du}{a^2+b^2(u+t)^2}-3.811g(\gamma_1)-\frac{0.045}{ab}.
$$
The proof follows by collecting the upper and lower bounds of $S_1$ and $S_2$.
\end{proof}

\section{Proof of Theorem \ref{Main thm}}\label{proof}
In this section we prove the Theorem \ref{Main thm}.

\begin{proof}[Proof of Theorem \ref{Main thm}]
Let $1/2<\sigma<1$. Since $0<(\sigma-1/2)^2<1/4$, we have that

\begin{align}\label{1}
\sum_{\rho=1/2+i\gamma}\frac{\sigma-1/2}{1/4+(t-\gamma)^2}
<\Re\sum_{\rho=1/2+i\gamma}\frac{1}{s-\rho}<
\sum_{\rho=1/2+i\gamma}\frac{(\sigma-1/2)^{-1}}{1+4(t-\gamma)^2}.
\end{align}

Recall that $c$ denotes the part of zeros of $\zeta(s)$ on the line $1/2+it$. Then, the total quantity $N(T)$ of non-trivial zeros
of $\zeta(s)$ in the rectangle $0<\sigma<1,\, 0<t<T$ can be expressed as $N(T)=cN(T)+(1-c)N(T)$.
Then,
by Lemma \ref{sumzero} and Lemma \ref{number of zeros} with $cN(T)$

\begin{align}\label{2}
&\Re\sum_{\rho=1/2+i\gamma}\frac{1}{s-\rho}=\sum_{\rho=1/2+i\gamma}\frac{\sigma-1/2}{(\sigma-1/2)^2+(t-\gamma)^2}\\ \nonumber
&=\frac{c(\sigma-1/2)}{2\pi}\int_{\gamma_1}^{\infty}
\left(\frac{\log(u/2\pi)}{(\sigma-1/2)^2+(u-t)^2}+\frac{\log(u/2\pi)}{(\sigma-1/2)^2+(u+t)^2}\right)du\\
&+c(\sigma-1/2)M(t),\nonumber
\end{align}
where $M(t)=O(\log t)$ as $t\to\infty$ and the explicit lower and upper bounds of $M(t)$ for $t>14.134725\ldots$ are given in Lemma \ref{sumzero}.

    Combining (\ref{1}) and (\ref{2}) and applying Lemmas \ref{int}, \ref{int1} and \ref{sumzero} with $a=1/2,\, b=1$ and $\alpha=\gamma_1=14.134725 \ldots$ for the lower bound we get
\begin{align*}
&\Re\sum_{\rho=1/2+i\gamma}\frac{1}{s-\rho}
>\frac{c(\sigma-1/2)}{2\pi}\int_{\gamma_1}^{\infty}
\left(\frac{\log(u/2\pi)}{1/4+(u-t)^2}+\frac{\log(u/2\pi)}{1/4+(u+t)^2}\right)du\\
&+c(\sigma-1/2)\left(-0.88\log t-2.32\log\log t-18.41-\frac{1.465}{t}-\frac{3.811}{0.25+(\gamma_1+t)^2}\right)\\
&>c(\sigma-1/2)\left(0.12\log\frac{t}{2\pi}-2.32\log\log t
-18.432-\epsilon_1(t)\right),\\
\end{align*}
where 
\begin{align}\label{eq:eps1}
\epsilon_1(t)=\left(\frac{1}{8\pi t}-\frac{1}{2\pi(t-\gamma_1)}\right)
\log\frac{t}{2\pi}-\frac{1.465}{t}-\frac{\gamma_1\log \frac{\gamma_1}{(2\pi)}}{2\pi t^2}-\frac{3.811}{0.25+(\gamma_1+t)^2}.
\end{align}
We check with Mathematica \cite{Mathematica} that 
\begin{align*}
0.12\log\frac{t}{2\pi}-2.32\log\log t-18.432\geqslant 49\times10^{-6},
\,|\epsilon_1(t)|\leqslant1.65\times 10^{-113},
\end{align*}
when $t\geqslant1.984\times10^{114}$.

By the same arguments, with $a=1$ and $b=2$, for the upper bound we get
\begin{align*}
&\Re\sum_{\rho=1/2+i\gamma}\frac{1}{s-\rho}
<\frac{c(\sigma-1/2)}{2\pi}\int_{\gamma_1}^{\infty}
\left(\frac{\log(u/2\pi)}{1+4(u-t)^2}+\frac{\log(u/2\pi)}{1+4(u+t)^2}\right)du
\end{align*}
\begin{align*}
&+c(\sigma-1/2)\left(0.22\log t +0.58\log\log t+4.603+\frac{0.367}{t}+\frac{3.811}{1+4(\gamma_1+t)^2}\right)\\
&<c(\sigma-1/2)\left(0.49\log\frac{t}{2\pi}+0.58\log\log t+4.603+\epsilon_2(t)\right),\\
\end{align*}
where
\begin{align}\label{eq:eps2}
\epsilon_2(t)=\frac{0.637}{t}+\frac{3.811}{1+4(t+\gamma_1)^2}+\frac{\log(t+1)+\frac{1}{2}\log\frac{2t^3}{4\pi^3}}{8\pi t}.
\end{align}
\end{proof}

\section{Can $\Re\frac{\xi'}{\xi}(s)$ remain positive if there are zeros off the critical line?}\label{Can?}

In this section we assume that the Riemann hypothesis fails by three different scenarios:\\
I. - there is only one zero in the region $1/2<\sigma<1$, $t>0$,\\
II. - there is a finite number $n\geqslant2$ of zeros off the critical line,\\
III. - there are infinitely many of zeros off the critical line.

\bigskip

     I. Assume that there is one point $\tilde{\beta}+i\tilde{\gamma}$ such that $\zeta(\tilde{\beta}+i\tilde{\gamma})=0$ for $1/2<\tilde{\beta}<1$, $t>0$.
Then, by Theorem \ref{Main thm} with $c=1$ and estimation,
\begin{align*}
&\Re\frac{\xi'}{\xi}(s)=
\left(\sigma-\frac{1}{2}\right)\sum_{\rho=1/2+i\gamma}\frac{1}{(\sigma-1/2)^2+(t-\gamma)^2}\\ \nonumber
&+\frac{\sigma-\tilde{\beta}}{(\sigma-\tilde{\beta})^2+(t-\tilde{\gamma})^2}
+\frac{\sigma-\tilde{\beta}}{(\sigma-\tilde{\beta})^2+(t+\tilde{\gamma})^2}\\ \nonumber
&+\frac{\sigma-(1-\tilde{\beta})}{(\sigma-(1-\tilde{\beta}))^2+(t-\tilde{\gamma})^2}
+\frac{\sigma-(1-\tilde{\beta})}{(\sigma-(1-\tilde{\beta}))^2+(t+\tilde{\gamma})^2}\\
&>0.11\left(\sigma-\frac{1}{2}\right)\log\frac{t}{2\pi}+\frac{\sigma-\tilde{\beta}}{(\sigma-\tilde{\beta})^2+(t-\tilde{\gamma})^2}
+O\left(\frac{\log \log t}{\log t}\right)>0
\end{align*}
if
\begin{equation}\label{ineq}
(\sigma,t)\in\left\{\frac{\sigma-\tilde{\beta}}{(\sigma-\tilde{\beta})^2+(t-\tilde{\gamma})^2}>-0.11\left(\sigma-\frac{1}{2}\right)
\log\frac{t}{2\pi}\right\}
\end{equation}
and $t$ is sufficiently large that $\log\log t/\log t $ is negligible. The region of $(\sigma,t)$ given by (\ref{ineq}) might have the following gray view given in Figure 1 below.
 Figure 1 was obtained by Mathematica \cite{Mathematica} with some chosen point $\tilde{\beta}+i\tilde{\gamma}$.
\begin{figure}[h]
    \centering
    \begin{minipage}{0.45\textwidth}
        \centering
        \includegraphics[width=0.9\textwidth]{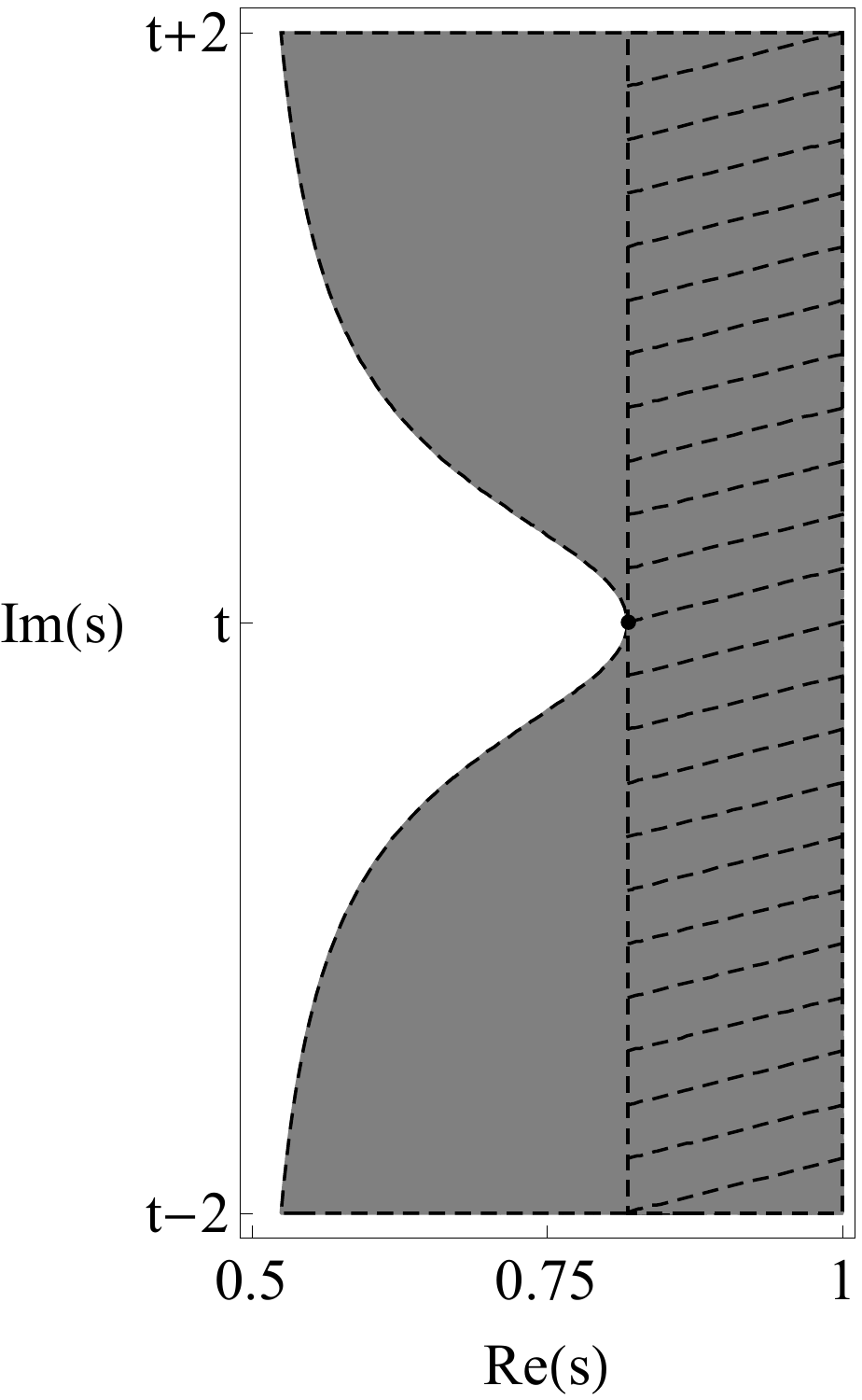} 
        \caption{Whole gray region satisfies inequality (\ref{ineq}). Theorem \ref{S-D} or \ref{MSZ} gives a dashed gray strip, where $\Re\xi'/\xi(s)>0$.}
    \end{minipage}\hfill
    \begin{minipage}{0.45\textwidth}
        \centering
        \includegraphics[width=0.9\textwidth]{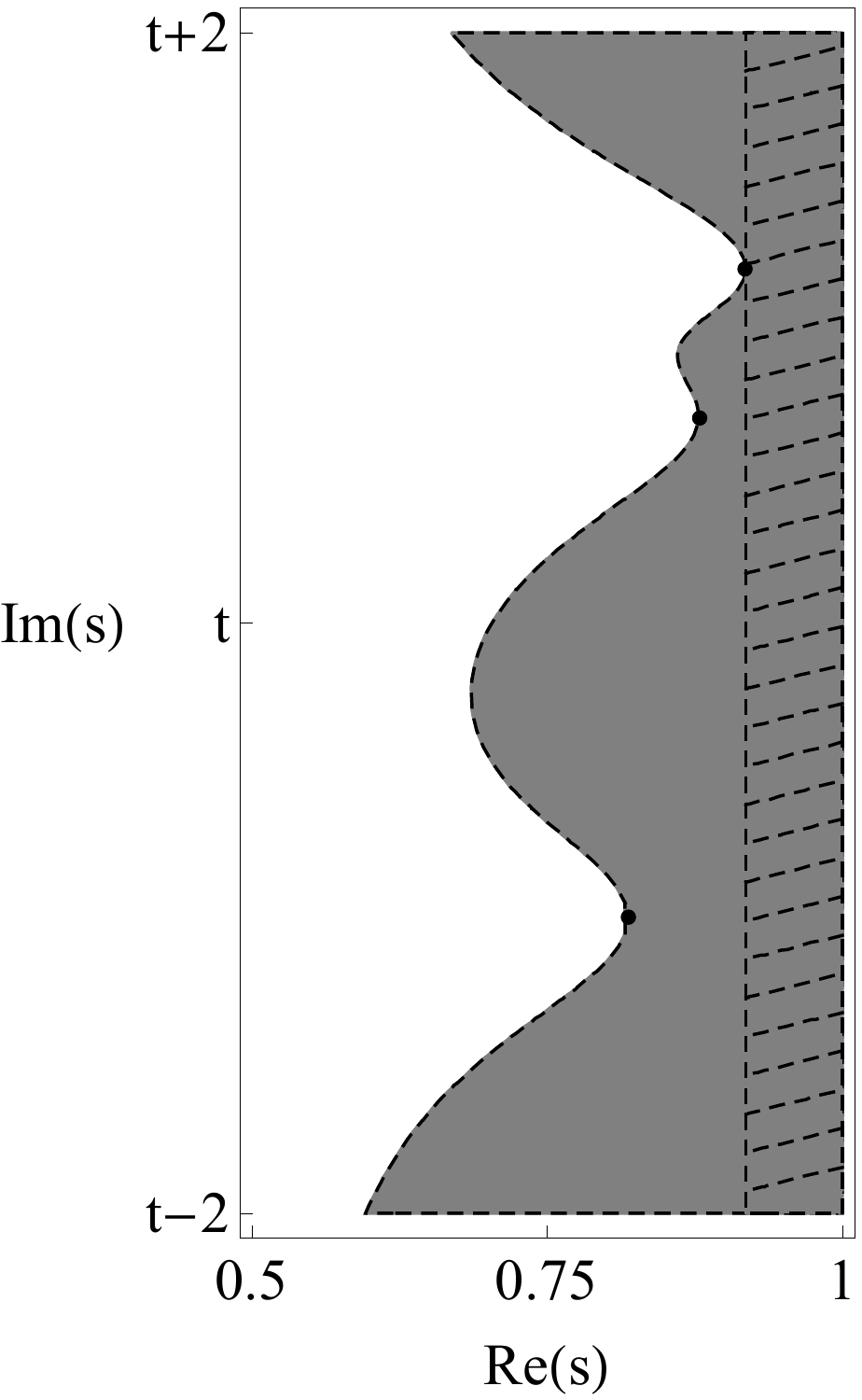} 
        \caption{Whole gray region satisfies inequality (\ref{ineq1}). Theorem \ref{S-D} or \ref{MSZ} gives a dashed gray strip, where $\Re\xi'/\xi(s)>0$.}
    \end{minipage}
\end{figure}

\bigskip

II. Assume that there is a finite number $n\geqslant2$ of points $\tilde{\beta}_k+i\tilde{\gamma}_k$, $k=1,\,\ldots,\,n$ such that $\zeta(\tilde{\beta}_k+i\tilde{\gamma}_k)=0$ for $1/2<\tilde{\beta}_k<1$, $t>0$.
Then, by Theorem \ref{Main thm} with $c=1$ and previous means,
\begin{align*}
&\Re\frac{\xi'}{\xi}(s)
>0.11\left(\sigma-\frac{1}{2}\right)\log\frac{t}{2\pi}
+\sum_{k=1}^{n}\frac{\sigma-\tilde{\beta}_k}{(\sigma-\tilde{\beta}_k)^2+(t-\tilde{\gamma}_k)^2}
+O\left(\frac{\log \log t}{\log t}\right)>0
\end{align*}
if
\begin{equation}\label{ineq1}
(\sigma,t)\in \left\{\sum_{k=1}^{n}\frac{\sigma-\tilde{\beta}_k}{(\sigma-\tilde{\beta}_k)^2+(t-\tilde{\gamma}_k)^2}
>-0.11\left(\sigma-\frac{1}{2}\right)\log\frac{t}{2\pi}\right\}
\end{equation}
and $t$ is sufficiently large that $\log\log t/\log t$ is negligible. The region of $(\sigma,t)$ given by (\ref{ineq1}) might have the following gray view given in Figure 2 above.
Figure 2 was obtained by Mathematica \cite{Mathematica} too with some chosen $\tilde{\beta}_k$ and $\tilde{\gamma}_k$, where the black points are $\tilde{\beta}_k+i\tilde{\gamma}_k$.

\bigskip

    III. Assume that there are infinitely many points $\tilde{\beta}_k+i\tilde{\gamma}_k$, such that $\zeta(\tilde{\beta}_k+i\tilde{\gamma}_k)=0$ for $1/2<\tilde{\beta}_k<1$, $t>0$. 
    
Then, by the same arguments as in I. and II.,
\begin{align}\label{inf}
&\Re\frac{\xi'}{\xi}(s)>c\cdot 0.11\left(\sigma-\frac{1}{2}\right)
\log \frac{t}{2\pi} \nonumber \\
&+\sum_{\substack{\tilde{\rho}=\tilde{\beta}_k+i\tilde{\gamma}_k\\\tilde{\gamma}_k>0}}\frac{\sigma-\tilde{\beta}_k}{(\sigma-\tilde{\beta}_k)^2+(t-\tilde{\gamma}_k)^2}
-\sum_{\tilde{\gamma}_k>0}\frac{1/2}{(t+\tilde{\gamma}_k)^2}
+O\left(\frac{\log \log t}{\log t}\right)>0
\end{align}
if
\begin{equation*}\label{ineq2}
(\sigma,t)\in \left\{
\sum_{\substack{\tilde{\rho}=\tilde{\beta}_k+i\tilde{\gamma}_k\\\tilde{\gamma}_k>0}}\frac{(\sigma-\tilde{\beta}_k)}{(\sigma-\tilde{\beta}_k)^2+(t-\tilde{\gamma}_k)^2}
>-c\cdot0.11\left(\sigma-\frac{1}{2}\right)\log\frac{t}{2\pi}
\right\}
\end{equation*}
and $t$ is sufficiently large. We note that
$\sum_{\tilde{\gamma}_k>0}\frac{1/2}{(t+\tilde{\gamma}_k)^2}=O\left(\frac{\log t}{t}\right),\,t\to\infty$ in (\ref{inf}), see Lemmas \ref{int1} and \ref{sumzero}.






\begin{thebibliography}{20}

\bibitem{Apostol}{\sc Tom M. Apostol}, {\it Introduction to Analytic Number Theory}, Springer, 1998.

\bibitem{Broughan} {\sc K. Broughan}, {\it Extension of the Riemann $\xi$-Function's Logarithmic Derivative Positivity Region to Near the Critical Strip}, Canad. Math. Bull., 52(2), 186--194, 2009, doi:10.4153/CMB-2009-021-3

\bibitem{Conrey}{\sc J.B. Conrey}, {\it More than two fifths of the zeros of the Riemann zeta function are on the critical line}, J. Reine Angew. Math., 399 (1989) 1-26.

\bibitem{Edwards}{\sc H.M. Edwards}, {\it Riemann's Zeta Function}, Academic Press, New York
(1974). Reprinted by Dover Publications, Mineola, N.Y. (2001).

\bibitem{Feng}{\sc S. Feng}, {\it Zeros of the Riemann zeta function on the critical line}, J. Number Theory, 132(4), 2012, 511-542.


\bibitem{RG}{\sc R. Garunk\v{s}tis}, {\it On a positivity property of the Riemann $\xi$-function}, Liet. matem. rink. 42 (2002), 179--184.

\bibitem{Hardy}{\sc G. H. Hardy}, {\it Sur les z\'{e}ros de la fonction $\zeta(s)$}. Comptes Rendus 158(1914) 1012--1014.

\bibitem{Has et al.}{\sc E. Hasanalizade, Q. Shen, P. J. Wong},
{\it Counting zeros of the Riemann zeta function}, J. Number Theory, 2021.

\bibitem{Hinkkanen}{\sc A. Hinkkanen}, {\it On functions of bounded type}, Complex Variables Theory Appl.
34 (1997), 119-139.


\bibitem{Lagarias}{\sc J. C. Lagarias}, {\it On a positivity property of the Riemann $\xi$ function}, Acta Arith. 89 (1999), 217--234.

\bibitem{Levinson}{\sc N. Levinson}, {\it More than one third of the zeros of Riemann's zeta-function are on $\sigma = 1/2$}, Adv. Math., 13 (4) (1974), 383-436.

\bibitem{Vigirdas}{\sc V. Mackevičius}, {\it Integralas ir matas}, TEV, 1998, ISBN 9986-546-47-8

\bibitem{Mathematica} Mathematica (Version 9.0), Wolfram Research, Inc., Champaign, Illinois, 2012

\bibitem{MSZ} {\sc Yu. Matiyasevich, F. Saidak, P. Zvengrowski},
{\it Horizontal monotonicity of the modulus of the zeta function, L-functions, and related functions
}, Acta Arith., 166.2 (2014), 189--200.

\bibitem{Murty}{\sc M. R. Murty}, {\it Problems in Analytic Number Theory}, Grad. Texts in Math. 206,
Springer, New York, 2001.



\bibitem{Platt_Trudgian}{\sc D. J. Platt, T. S. Trudgian}, {\it
An improved explicit bound on $|\zeta(1/2+it)|$},
J. Number Theory, 147, 2015, 842--851,

\bibitem{Pratt et al.} {\sc K. Pratt, N. Robles, A. Zaharescu et al.}, {\it More than five-twelfths of the zeros of $\zeta$ are on the critical line}, Res. Math. Sci. 7(2), (2020).

\bibitem{Selberg}{\sc A. Selberg}, {\it On the zeros of Riemann's zeta-function}, Skr. Norske Vid. Akad. Oslo I., 10 (1942), 59 p.

\bibitem{Sondow}{\sc J. Sondow, C. Dumitrescu}, {\it A monotonicity property of Riemann's xi function
and a reformulation of the Riemann hypothesis}, Period. Math. Hungar. 60 (2010), 37--40.

\bibitem{Spivak} {\sc M. Spivak}, {\it Calculus (3 ed.)}, Houston, Texas: Publish or Perish, 1994, ISBN 978-0-914098-89-8.



\bibitem{titch_R}{\sc E.C. Titchmarsh}, {\it The Theory of the Riemann Zeta-Function}, 2nd edition, Oxford Univ. Press, 1986.

\bibitem{Trudgian}{\sc T.S. Trudgian}, {\it An improved upper bound for the argument of the Riemann zeta-function on the critical line II}, Journal of Number Theory, 134 (2014),
280-292.



\bibitem{MathWorld}Wolfram MathWorld, \hyperref[https://mathworld.wolfram.com/Xi-Function.html]{https://mathworld.wolfram.com/Xi-Function.html}

\end{thebibliography}
\end{document}